\documentclass[11pt]{article}
\usepackage[a4paper,margin=28mm]{geometry}
\usepackage{amsmath,amssymb,amsthm,aliascnt}
\usepackage{titlesec}
\titlespacing*{\section}{0pt}{2.5ex plus 1ex minus .2ex}{1.2ex plus .2ex}
\titlespacing*{\subsection}{0pt}{2ex plus .5ex minus .2ex}{.8ex plus .2ex}
\allowdisplaybreaks[4]
\usepackage{xcolor}
\usepackage{hyperref}
\hypersetup{colorlinks=true,linkcolor=cyan,citecolor=red,urlcolor=black,filecolor=black}
\numberwithin{equation}{section}
\newtheorem{theorem}{Theorem}[section]
\newaliascnt{lemma}{theorem}
\newtheorem{lemma}[lemma]{Lemma}
\aliascntresetthe{lemma}
\newaliascnt{remark}{theorem}

\aliascntresetthe{remark}
\newaliascnt{proposition}{theorem}
\newtheorem{proposition}[proposition]{Proposition}
\aliascntresetthe{proposition}

\begin{document}
\begin{center}
{\Large\bf
Real-rootedness and ultra log-concavity of rank-two matroid Ehrhart $h^*$-polynomials
}
\end{center}

\begin{center}
Houshan Fu\par

School of Mathematics and Information Science, Guangzhou University\\
Guangzhou 510006, Guangdong, P. R. China\par
 
Email address: fuhoushan@gzhu.edu.cn
\end{center}

\begin{abstract}
We prove that the Ehrhart $h^*$-polynomial of a rank-two matroid with exactly three parallel classes is real-rooted whenever its smallest parallel class has size at most three. This bound is sharp: the loopless rank-two matroids with parallel-class sizes $(4,561,600)$ and $(4,a,a+29)$, for all sufficiently large integers $a$, have $h^*$-polynomials that are not real-rooted. By matroid duality, these examples disprove Ferroni's real-rootedness conjecture even for cycle matroids of theta graphs. Nevertheless, every matroid of rank two or corank two has a positive $h^*$-coefficient sequence that is ultra log-concave of order equal to the polynomial's degree. In particular, the unimodality conjecture holds in both cases. \par\smallskip \noindent\textbf{Keywords:} Matroid base polytope; $h^*$-polynomial; real-rootedness; ultra log-concavity \par\smallskip \noindent\textbf{Mathematics Subject Classification:} 05B35, 52B20, 26C10
\end{abstract}

\section{Introduction}
The Ehrhart theory of matroid base polytopes relates lattice-point enumeration to the combinatorics of bases. We follow Oxley~\cite{Oxley} for matroid terminology. Let $M$ be a matroid on a finite ground set $E$, with set of bases $\mathcal B(M)$ and rank $\operatorname{rk}(M)$, the common cardinality of its bases. Its \emph{base polytope} is
\[
 P(M):=\operatorname{conv}\{\mathbf e_B:B\in\mathcal B(M)\}\subseteq\mathbb R^E.
\]
Here $\mathbf e_B:=\sum_{e\in B}\mathbf e_e$, where $\mathbf e_e$ is the standard unit vector indexed by $e$. Write $m=|E|$ and $n=\dim P(M)$. Ehrhart's theorem~\cite{Ehrhart} states that $\operatorname{ehr}(M,q):=|qP(M)\cap\mathbb Z^E|$, for integers $q\ge0$, agrees with a polynomial of degree $n$, its \emph{Ehrhart polynomial}. The \emph{Ehrhart $h^*$-polynomial} $h_M^*(t)$ is defined by
\[
 \sum_{q\ge0}\operatorname{ehr}(M,q)t^q
   =\frac{h_M^*(t)}{(1-t)^{n+1}}.
\]
Write $h_M^*(t)=\sum_{j=0}^d h_j^*t^j$, where $d=\deg h_M^*\le n$. Stanley's nonnegativity theorem~\cite[Theorem 2.1]{Stanley1980} asserts that $h_j^*\in\mathbb Z_{\ge0}$ for every $j$. The constant term is $h_0^*=1$.

An element of $M$ is a \emph{loop} if it belongs to no basis. Two distinct nonloop elements are \emph{parallel} if they form a circuit. The \emph{parallel classes} are the equivalence classes of nonloop elements under the relation of equality or parallelism. Deleting loops does not change $h_M^*$, since their coordinates vanish throughout $P(M)$. A loopless rank-two matroid is determined, up to isomorphism, by its parallel-class sizes: its bases are precisely the pairs of elements from distinct parallel classes.

De Loera, Haws, and K\"oppe~\cite[Conjecture 2(A)]{DLHK} conjectured that the coefficients of $h_M^*$ are \emph{unimodal} for every matroid $M$: they increase weakly to a maximum and then decrease weakly. They established unimodality for uniform matroids of rank two~\cite[Theorem 3(1)]{DLHK}. Ferroni~\cite[Conjecture 1.3]{Ferroni} subsequently proposed a stronger conjecture: every matroid Ehrhart $h^*$-polynomial is \emph{real-rooted}, that is, all of its zeros are real. For polynomials with nonnegative coefficients, real-rootedness implies unimodality.

Known real-rooted families include minimal matroids \cite[Corollary 1.7]{Ferroni} and rank-two sparse paving matroids \cite[Theorem 1.3]{FJS}. In rank two, minimal matroids have parallel-class sizes $(1,1,b)$, where $b\ge1$, whereas the loopless sparse paving matroids are precisely those with every parallel class of size at most two. Neither family covers general rank-two matroids. For uniform matroids, Adiprasito and Zhang~\cite[Theorems 1.1 and 3.1]{AdiprasitoZhang} proved real-rootedness for rank-three hypersimplices and eventual real-rootedness for hypersimplices of each fixed rank.

The separate Ehrhart-positivity conjecture of De Loera, Haws, and K\"oppe~\cite[Conjecture 2(B)]{DLHK} concerns the coefficients of $\operatorname{ehr}(M,q)$ in the monomial basis, not the $h^*$-coefficients. Ehrhart positivity holds for hypersimplices~\cite[Theorem 1.2]{FerroniHypersimplex}, lattice path matroids~\cite[Theorem 1.2]{FerroniMoralesPanova}, and all rank-two matroids~\cite[Theorem 1.2]{FJS}, but fails in every rank at least three~\cite[Theorem 1.7]{FerroniNonpositive}. These positivity results do not settle the real-rootedness of $h_M^*$. See also~\cite[Remark 8.5]{FerroniNonpositive}.

This paper addresses two questions: whether real-rootedness can fail in rank two, and whether the $h^*$-coefficients are nevertheless ultra log-concave for every rank-two matroid. The first theorem gives real-rooted families and counterexamples among cycle matroids of theta graphs. Cycle matroids of theta graphs are dual to loopless rank-two matroids having exactly three parallel classes.

The \emph{dual} $M^*$ has bases $E\setminus B$ for $B\in\mathcal B(M)$, and the \emph{corank} of $M$ is $\operatorname{rk}(M^*)=m-\operatorname{rk}(M)$. The map $\mathbf x\mapsto\mathbf1-\mathbf x$, where $\mathbf1\in\mathbb R^E$ is the all-one vector, is an affine lattice isomorphism from $P(M)$ to $P(M^*)$. Hence
\begin{equation}\label{eq:duality}
 h_M^*(t)=h_{M^*}^*(t).
\end{equation}
For positive integers $a,b,c$, let $\Theta(a,b,c)$ denote the \emph{theta graph} formed by three paths of lengths $a,b,c$ with two common, distinct endpoints and pairwise disjoint internal vertices. Parallel edges are allowed when path lengths equal one. For a graph $G$, its \emph{cycle matroid} $M(G)$ has the edges of $G$ as its ground set and the edge sets of maximal forests as its bases. The complements of the spanning trees of a theta graph are precisely the pairs of edges on distinct paths. Thus $M(\Theta(a,b,c))$ has corank two, and its dual $M(\Theta(a,b,c))^*$ is the connected rank-two matroid with parallel-class sizes $(a,b,c)$.
\begin{theorem}\label{thm:real-rootedness}
\leavevmode
\begin{itemize}
\item[$\mathrm{(1)}$] If $1\le c\le3$ and $c\le a\le b$, then
$h_{M(\Theta(c,a,b))}^*(t)$ is real-rooted, with simple negative zeros
whenever it is nonconstant.
\item[$\mathrm{(2)}$] The polynomial
$h_{M(\Theta(4,561,600))}^*(t)$ is not real-rooted.
\item[$\mathrm{(3)}$] For all sufficiently large integers $a$, $h_{M(\Theta(4,a,a+29))}^*(t)$ is not real-rooted.
\end{itemize}
\end{theorem}

By loop deletion and \eqref{eq:duality}, part (1) also applies to rank-two matroids with exactly three parallel classes whose smallest class has size at most three. Parts (2) and (3) show that this bound is sharp. The counterexamples in these two parts are cycle matroids of simple series-parallel graphs. The same method also proves real-rootedness for rank-two matroids with parallel-class sizes $(1,1,a,b)$, as stated in \autoref{thm:four-classes}.

The counterexamples leave open the unimodality question. The next theorem establishes the stronger property of ultra log-concavity in rank two and corank two. To state the coefficient result, recall that a positive sequence $(\gamma_j)_{j=0}^d$ is \emph{log-concave} if $\gamma_j^2\ge\gamma_{j-1}\gamma_{j+1}$ for $1\le j<d$. It is \emph{strictly log-concave} if all these inequalities are strict. It is \emph{ultra log-concave of order $d$} if $(\gamma_j/\binom dj)_{j=0}^d$ is log-concave, or equivalently,
\begin{equation}\label{eq:ulc}
 j(d-j)\gamma_j^2\ge(j+1)(d-j+1)\gamma_{j-1}\gamma_{j+1}
 \qquad(1\le j<d).
\end{equation}
Throughout this paper, the normalization uses $d=\deg h_M^*$, not the dimension $n$.

\begin{theorem}\label{thm:rank-two-full-unimodal}
Let $M$ be a matroid of rank two or corank two. The coefficient sequence $(h_j^*)_{j=0}^d$ of $h_M^*(t)$ is positive and ultra log-concave of order $d$. In particular, it is unimodal and strictly log-concave at each internal index.
\end{theorem}

The rank-two duals of the counterexamples in \autoref{thm:real-rootedness}(2)--(3) are Ehrhart positive~\cite[Theorem 1.2]{FJS}. Their $h^*$-coefficient sequences are ultra log-concave by \autoref{thm:rank-two-full-unimodal}, although their $h^*$-polynomials are not real-rooted by \eqref{eq:duality}.

The base polytope of a loopless rank-two matroid is the edge polytope of the complete multipartite graph whose parts are its parallel classes. Ohsugi and Hibi~\cite[Theorem 2.6]{OhsugiHibi} obtained an explicit Ehrhart polynomial formula for these polytopes.

For parallel-class sizes $(1,a,b)$, the base polytope is unimodularly equivalent to an order polytope~\cite[Proposition 2.2(2)]{HMConic}. Its $h^*$-polynomial is the non-nesting rook polynomial of an $a$-row, $b$-column rectangle with the upper-left cell removed, by \eqref{eq:theta-one} and \cite[Example 2.3]{AlexanderssonJal}. The upper-left cell can be occupied only in a one-rook placement, so deleting it subtracts $t$ from the rectangle's non-nesting rook polynomial. For $2\le a\le b$, the remaining board has $a$ rows and its polynomial has degree $a$, so ultra log-concavity of order $a$ for this subfamily also follows from \cite[Corollary 4.1]{AlexanderssonJal}.

Jiang's combinatorial formula for connected positroid $h^*$-polynomials~\cite[Theorem 1.1]{Jiang} also applies to connected rank-two matroids, which are positroids up to relabelling \cite[Corollary 6.2]{FJS}.

Building on the rank-two formula of Ferroni, Jochemko, and Schr\"oter~\cite[Proposition 5.1]{FJS}, our key observation is a subset-counting interpretation of its correction coefficients. The resulting binomial-product identities lead to Jacobi-polynomial expressions for proving real-rootedness and to coefficient estimates that establish ultra log-concavity.

The paper is organized as follows. \autoref{sec:coefficients} develops the coefficient identities. \autoref{sec:zeros} proves the real-rootedness results, and \autoref{sec:finite-infinite} constructs the counterexamples. \autoref{sec:rank-two} establishes ultra log-concavity in rank two and corank two.

\section{The rank-two coefficient formula}\label{sec:coefficients}
This section recalls the rank-two coefficient formula and derives the counting identities used in the proofs of both main theorems.

We use the convention $\binom{k}{j}=0$ for integers $k\ge0$ and $j\notin\{0,\ldots,k\}$. For a polynomial $p(t)$, the notation $[t^j]p(t)$ denotes its coefficient of $t^j$. As noted in the introduction, it suffices to work with loopless matroids.

For integers $0\le r\le m$, the uniform matroid $U_{r,m}$ has as its bases all $r$-element subsets of an $m$-element ground set. For matroids $M,N$ on disjoint ground sets, their direct sum $M\oplus N$ has bases $B_M\cup B_N$, where $B_M\in\mathcal B(M)$ and $B_N\in\mathcal B(N)$. A matroid is \emph{connected} if it is not the direct sum of two matroids on nonempty ground sets.

A loopless rank-two matroid with $s$ parallel classes is connected if and only if $s\ge3$: any two distinct elements in the same class form a two-element circuit, and representatives of three distinct classes form a circuit. For $s=2$, the matroid is $U_{1,a_1}\oplus U_{1,a_2}$. We first consider the connected case. Let the parallel-class sizes be $a_1,\ldots,a_s$, with $\sum_i a_i=m$. Ferroni, Jochemko, and Schr\"oter~\cite[Proposition 5.1]{FJS} proved that
\begin{equation}\label{circle:eq:fjs}
 h_M^*(t)=\sum_{j\ge0}\binom m{2j}t^j-\sum_{i=1}^{s} p^*_{a_i,m}(t),\qquad
 p^*_{a,m}(t):=\sum_{j\ge1}\sum_{k=j}^a
            \binom kj\binom{m-k-1}{j-1}t^j.
\end{equation}

To make the cancellations in \eqref{circle:eq:fjs} explicit, we interpret each correction coefficient as a count of subsets having more than half their elements in a prescribed subset. For an integer $0\le a<m$, fix an $(a+1)$-element subset $X$ of an $m$-element set $E$. For $j\ge1$, we have the counting formula
\begin{equation}\label{eq:majority}
 C_j(a,m):=[t^j]p^*_{a,m}(t)
 =\#\{S\subseteq E:|S|=2j,\ |S\cap X|\ge j+1\}.
\end{equation}
Indeed, fix a linear order on $E$ in which the elements of $X$ occupy the first $a+1$ positions. If the $(j+1)$st selected element has position $k+1$, there are $\binom kj\binom{m-k-1}{j-1}$ choices for the other elements. Summing over $j\le k\le a$ proves \eqref{eq:majority}. We also set $C_0(a,m)=0$.

For an integer $r$ with $0\le r\le m-2$, partition an $m$-element set into blocks of sizes $r+1$ and $m-r-1$. A $2j$-element subset has either more than $j$ elements in one block, or exactly $j$ in each. Thus
\[
\binom m{2j}-C_j(r,m)-C_j(m-r-2,m)=\binom{r+1}{j}\binom{m-r-1}{j}.
\]
To encode the binomial products in a single polynomial, define, for nonnegative integers $a,b$,
\[
 J_{a,b}(t):=\sum_{j\ge0}\binom aj\binom bj t^j.
\]
Multiplying the coefficient identity above by $t^j$ and summing over $j\ge0$ yields
\begin{equation}\label{eq:block-identity}
 \sum_{j\ge0}\binom m{2j}t^j -p^*_{r,m}(t)-p^*_{m-r-2,m}(t)=J_{r+1,m-r-1}(t) \qquad(0\le r\le m-2).
\end{equation}

The same counting argument applies when the two subsets overlap in one element. For an integer $r$ with $0\le r<m$, choose $U,V\subseteq E$ with $|U|=r+1$, $|V|=m-r$, and $U\cap V=\{e\}$, so $U\cup V=E$. A $2j$-element subset of $E$ cannot have more than $j$ elements in both $U$ and $V$. It has at most $j$ in each precisely when it avoids $e$ and has $j$ elements in each of $U\setminus\{e\}$ and $V\setminus\{e\}$. Counting the remaining subsets and summing over $j$ therefore gives
\begin{equation}\label{eq:overlap-identity}
 \sum_{j\ge0}\binom m{2j}t^j
 -p^*_{r,m}(t)-p^*_{m-r-1,m}(t)=J_{r,m-r-1}(t)
 \qquad(0\le r<m).
\end{equation}
We use \eqref{eq:block-identity} and \eqref{eq:overlap-identity} to obtain the Jacobi-polynomial expressions in \autoref{sec:zeros} and to evaluate the boundary terms in the coefficient estimates of \autoref{sec:rank-two}.

To adjust the indices of the correction terms, we also record their consecutive differences. For $1\le r<m$, subtracting the defining sums for $p^*_{r,m}$ and $p^*_{r-1,m}$ in \eqref{circle:eq:fjs} leaves only the terms with $k=r$:
\begin{equation}\label{eq:correction-increment}
 p^*_{r,m}(t)-p^*_{r-1,m}(t)
 =\sum_{j\ge1}\binom rj\binom{m-r-1}{j-1}t^j.
\end{equation}

\section{Real-rootedness}\label{sec:zeros}
As in \cite[Section 5.2]{FJS}, we prove real-rootedness by establishing alternating signs at suitable points. After a change of variable, the formulas in \autoref{lem:short-coefficients} relate the $h^*$-polynomials to Jacobi polynomials. We use the Jacobi--Gauss--Lobatto quadrature formula~\cite[Theorem 3.27]{STW} to obtain the bounds needed to preserve the signs at the Jacobi critical points.
\subsection{Binomial-product formulas}
The following identities express the relevant $h^*$-polynomials in terms of $J_{a,b}$, which we relate to Jacobi polynomials in the next subsection.

\begin{lemma}\label{lem:short-coefficients}
For positive integers $a,b$, the following identities hold:
\begin{align}
 h_{M(\Theta(1,a,b))}^*(t)&=J_{a,b}(t)-t,\label{eq:theta-one}\\
 h_{M(\Theta(2,a,b))}^*(t)&=J_{a+1,b+1}(t)-3t-(a+b-1)t^2,\label{eq:theta-two}
\end{align}
\begin{equation}
 h_{M(\Theta(3,a,b))}^*(t)=J_{a+2,b+2}(t)-tJ_{a+1,b+1}(t)-6t-(4a+4b-3)t^2-\binom{a+b-1}{2}t^3.\label{eq:theta-three}
\end{equation}
If $M$ is a rank-two matroid with parallel-class sizes $(1,1,a,b)$, then
\begin{equation}\label{eq:four-class-formula}
 h_M^*(t)=J_{a+1,b+1}(t)-2t.
\end{equation}
\end{lemma}
\begin{proof}
For $c\in\{1,2,3\}$, set $m=a+b+c$. The dual matroid $M(\Theta(c,a,b))^*$ is connected of rank two, with parallel-class sizes $(c,a,b)$. Thus \eqref{eq:duality} and \eqref{circle:eq:fjs} imply
\[
 h^*_{M(\Theta(c,a,b))}(t)
 =\sum_{j\ge0}\binom m{2j}t^j
  -p^*_{a,m}(t)-p^*_{b,m}(t)-p^*_{c,m}(t).
\]
For $c=1$, \eqref{eq:overlap-identity} with $r=a$ reduces the first three terms on the right to $J_{a,b}(t)$. Since $p^*_{1,m}(t)=t$, \eqref{eq:theta-one} follows.

When $m=a+b+2$, \eqref{eq:block-identity} with $r=a$ reads
\[
 \sum_{j\ge0}\binom m{2j}t^j-p^*_{a,m}(t)-p^*_{b,m}(t)
 =J_{a+1,b+1}(t).
\]
For $c=2$, subtracting $p^*_{2,m}(t)=3t+(m-3)t^2$ proves \eqref{eq:theta-two}. For parallel-class sizes $(1,1,a,b)$, the same identity and $2p^*_{1,m}(t)=2t$ give \eqref{eq:four-class-formula}.

Finally, let $c=3$, so $m=a+b+3$. Here \eqref{eq:block-identity} with $r=a$ involves $p^*_{b+1,m}$ rather than $p^*_{b,m}$. Therefore,
\[
 \sum_{j\ge0}\binom m{2j}t^j-p^*_{a,m}(t)-p^*_{b,m}(t)
 =J_{a+1,b+2}(t)+p^*_{b+1,m}(t)-p^*_{b,m}(t).
\]
By \eqref{eq:correction-increment}, the coefficient of $t^j$ on the right-hand side, for $j\ge0$, is
\[
 \binom{a+1}j\binom{b+2}j
 +\binom{a+1}{j-1}\binom{b+1}j
 =\binom{a+2}j\binom{b+2}j
 -\binom{a+1}{j-1}\binom{b+1}{j-1},
\]
where the equality follows from Pascal's identity. Hence
\[
 h^*_{M(\Theta(3,a,b))}(t)
 =J_{a+2,b+2}(t)-tJ_{a+1,b+1}(t)-p^*_{3,m}(t).
\]
As $p^*_{3,m}(t)=6t+(4m-15)t^2+\binom{m-4}{2}t^3$ by \eqref{circle:eq:fjs}, substituting $m=a+b+3$ proves \eqref{eq:theta-three}.
\end{proof}

\subsection{A bound at the critical points of Jacobi polynomials}
Fix integers $k\ge2$ and $\delta\ge0$. We use the degree-$k$ Jacobi polynomial \cite[equation (3.105)]{STW}
\[
 P_k^{(0,\delta)}(x) :=2^{-k}\sum_{j=0}^k\binom kj\binom{k+\delta}j
                 (x-1)^j(x+1)^{k-j}.
\]
For every $q(x)\in\mathbb R[x]$ with $\deg q<k$, orthogonality with respect to the weight $(1+x)^\delta$ gives~\cite[equation (3.88)]{STW}
\[
 \int_{-1}^{1}P_k^{(0,\delta)}(x)q(x)(1+x)^\delta\,dx=0.
\]
Set
\[
 \varphi(u):=P_k^{(0,\delta)}(1-2u).
\]
Substituting $x=1-2u$ in the defining sum expresses $\varphi$ in terms of $J_{a,b}$ from \autoref{sec:coefficients}:
\begin{equation}\label{eq:jacobi-transform}
 \varphi(u)=(1-u)^kJ_{k,k+\delta}\!\Bigl(-\frac{u}{1-u}\Bigr)
 \qquad(u\ne1).
\end{equation}
By orthogonality, the $k$ zeros of $\varphi$ are simple and lie in $(0,1)$ \cite[Theorem 3.2]{STW}. The zeros $0<u_1<\cdots<u_{k-1}<1$ of $\varphi'$ strictly interlace those of $\varphi$ by Rolle's theorem. The endpoint values are $\varphi(0)=1$ and $\varphi(1)=(-1)^k\binom{k+\delta}{k}$. Since $\varphi$ changes sign at each simple zero, interlacing gives
\begin{equation}\label{eq:jacobi-critical-signs}
 (-1)^i\varphi(u_i)>0\qquad(1\le i\le k-1).
\end{equation}
To obtain $k$ distinct zeros in $(0,1)$ after adding a polynomial $g$ of degree at most $k$, we preserve the alternating signs at $0,u_1,\ldots,u_{k-1},1$. When $g(0)=g(1)=0$, it suffices to show $|g(u_i)|<|\varphi(u_i)|$ at each critical point. The Jacobi--Gauss--Lobatto quadrature formula uses precisely these nodes.

To express its interior weights in terms of $\varphi(u_i)$, recall the Jacobi differential equation \cite[equations (3.89)--(3.91)]{STW}:
\[
 u(1-u)\varphi''(u)+[1-(\delta+2)u]\varphi'(u)
       +k(k+\delta+1)\varphi(u)=0.
\]
At $u_i$, the identity $\varphi'(u_i)=0$ reduces this equation to
\begin{equation}\label{eq:jacobi-critical-point}
 u_i(1-u_i)\varphi''(u_i)=-k(k+\delta+1)\varphi(u_i).
\end{equation}
For $x_i=1-2u_i$, the formulas in~\cite[equations (3.100), (3.132b), and (3.139c)]{STW}, together with \eqref{eq:jacobi-critical-point}, give the weight at $u_i$:
\[
 \frac1{2^{\delta+1}}
 \frac{2^{\delta+1}k(k+\delta+1)}
 {(1-x_i^2)^2[(P_k^{(0,\delta)})''(x_i)]^2}
 =\frac{k(k+\delta+1)}
 {u_i^2(1-u_i)^2\varphi''(u_i)^2}
 =\frac1{k(k+\delta+1)\varphi(u_i)^2}.
\]
Here $2^{-\delta-1}$ is the factor from $x=1-2u$, while $1-x_i^2=4u_i(1-u_i)$ and $\varphi''(u_i)=4(P_k^{(0,\delta)})''(x_i)$.

For every $f\in\mathbb R[u]$ with $\deg f\le2k-1$, the shifted quadrature formula reads~\cite[Theorem 3.27]{STW}
\begin{equation}\label{eq:jacobi-lobatto-rule}
 k(k+\delta+1)\int_0^1 f(u)(1-u)^\delta\,du
 =f(0)+\sum_{i=1}^{k-1}\frac{f(u_i)}{\varphi(u_i)^2}
       +\frac{(\delta+1)f(1)}{\varphi(1)^2}.
\end{equation}
The positivity of the quadrature weights allows us to control all critical-point ratios by a single integral, as the next lemma shows.

\begin{lemma}\label{lem:theta-square}
Let $g\in\mathbb R[u]$ satisfy $\deg g\le k$ and $g(0)=g(1)=0$. For $1\le i\le k-1$, we have
\begin{equation}\label{eq:theta-square-bound}
 \Bigl|\frac{g(u_i)}{\varphi(u_i)}\Bigr|^2
 \le k(k+\delta+1)u_i(1-u_i)
       \int_0^1\frac{g(u)^2}{u(1-u)}(1-u)^\delta du.
\end{equation}
If
\begin{equation}\label{eq:jacobi-perturbation-condition}
 \frac{k(k+\delta+1)}4
       \int_0^1\frac{g(u)^2}{u(1-u)}(1-u)^\delta du<1,
\end{equation}
then $\varphi+g$ has $k$ distinct zeros in $(0,1)$.
\end{lemma}
\begin{proof}
Since $u(1-u)$ divides $g(u)$, the quotient $p(u)=g(u)^2/[u(1-u)]$ is a polynomial of degree at most $2k-2$ with $p(0)=p(1)=0$. Substituting $f=p$ into \eqref{eq:jacobi-lobatto-rule} therefore eliminates the endpoint terms and yields
\begin{equation}\label{eq:shifted-lobatto}
 k(k+\delta+1)\int_0^1\frac{g(u)^2}{u(1-u)}(1-u)^\delta du
 =\sum_{i=1}^{k-1}\frac1{u_i(1-u_i)}
   \Bigl|\frac{g(u_i)}{\varphi(u_i)}\Bigr|^2.
\end{equation}
Each summand in \eqref{eq:shifted-lobatto} is nonnegative, so retaining one proves \eqref{eq:theta-square-bound}. Since $u_i(1-u_i)\le1/4$, \eqref{eq:theta-square-bound} and \eqref{eq:jacobi-perturbation-condition} imply $|g(u_i)|<|\varphi(u_i)|$ at every critical point. By \eqref{eq:jacobi-critical-signs} and $g(0)=g(1)=0$, the values of $\varphi+g$ at $0,u_1,\ldots,u_{k-1},1$ alternate in sign. The intermediate value theorem guarantees a zero in each of the $k$ intervals between consecutive nodes. The degree bound $\deg(\varphi+g)\le k$ then ensures that these are all its zeros and that each is simple.
\end{proof}

\subsection{Real-rooted families}\label{sec:positive-roots}
For $c=1,2$, the formulas in \autoref{lem:short-coefficients} allow a direct application of the perturbation criterion in \autoref{lem:theta-square}. For $c=3$, an adjacent-degree Jacobi identity first combines the two Jacobi terms into a positive multiple of $\varphi$ at its critical points, after which we apply the estimate in that lemma.

\begin{proof}[Proof of \autoref{thm:real-rootedness}(1)]
If $c=a=1$, then \eqref{eq:theta-one} gives $h_{M(\Theta(1,1,b))}^*(t)=1+(b-1)t$, including the constant case $b=1$. Assume otherwise and set
\[
 k=a+c-1,\qquad \delta=b-a,\qquad v=2k+\delta=a+b+2c-2.
\]
Thus $k\ge2$ and $\delta\ge0$. We use $\varphi$ and its critical points $u_i$ as in \autoref{lem:theta-square}. The parameters satisfy
\begin{equation}\label{eq:jacobi-parameter-bound}
 4k(k+\delta+1)=v(v+2)-\delta(\delta+2)\le v(v+2).
\end{equation}

\emph{The cases $c=1,2$.} Equations \eqref{eq:theta-one}, \eqref{eq:theta-two}, and \eqref{eq:jacobi-transform} imply
\[
 (1-u)^kh_{M(\Theta(c,a,b))}^*\!\Bigl(-\frac{u}{1-u}\Bigr)
 =\varphi(u)+g(u),\qquad
 g(u)=\begin{cases}
 u(1-u)^{k-1},&c=1,\\
 u(1-u)^{k-2}(3-vu),&c=2.
 \end{cases}
\]
Since $k\ge2$ for $c=1$ and $k\ge3$ for $c=2$, both perturbations satisfy the degree and endpoint hypotheses of \autoref{lem:theta-square}. It remains to verify the integral condition \eqref{eq:jacobi-perturbation-condition}. Direct integration, after expanding the square when $c=2$, gives
\begin{equation}\label{eq:short-perturbation-integrals}
 \int_0^1\frac{g(u)^2}{u(1-u)}(1-u)^\delta du
 =\begin{cases}\dfrac1{(v-2)(v-1)},&c=1,\\
 \dfrac3{(v-4)(v-1)},&c=2.
 \end{cases}
\end{equation}
For $c=1$, we use
\begin{equation}\label{eq:short-square-comparison}
 4\bigl((v-2)(v-1)-k(k+\delta+1)\bigr)
 =(v-4)(3v-2)+\delta(\delta+2)\ge0,
\end{equation}
which holds for all $k\ge2$ and $\delta\ge0$ because $v\ge4$. Equations \eqref{eq:short-perturbation-integrals} and \eqref{eq:short-square-comparison} show that the left-hand side of \eqref{eq:jacobi-perturbation-condition} is at most $1/4<1$. For $c=2$, we have $v\ge6$, so
\[
 16(v-4)(v-1)-3v(v+2)=(v-6)(13v-8)+16>0.
\]
Together with \eqref{eq:short-perturbation-integrals} and \eqref{eq:jacobi-parameter-bound}, this verifies \eqref{eq:jacobi-perturbation-condition} for $c=2$. In both cases, \autoref{lem:theta-square} implies that $\varphi+g$ has $k$ distinct zeros in $(0,1)$. Under $t=-u/(1-u)$, these correspond to $k$ distinct negative zeros of $h_{M(\Theta(c,a,b))}^*$, whose degree $d$ equals $k$ by \eqref{eq:theta-one} and \eqref{eq:theta-two}.

\emph{The case $c=3$.} Here $k=a+2\ge5$ and $v=a+b+4\ge10$. Equations \eqref{eq:theta-three} and \eqref{eq:jacobi-transform} yield
\begin{equation}\label{eq:theta-three-transform}
 \widetilde H(u):=(1-u)^kh_{M(\Theta(3,a,b))}^*\!\Big(-\frac{u}{1-u}\Big)
 =\varphi(u)+uP_{k-1}^{(0,\delta)}(1-2u)+g(u),
 \end{equation}
where
\[
 g(u)=u(1-u)^{k-3}
 \Bigl(6-(4v-7)u+\frac{v^2-3v+4}{2}u^2\Bigr).
\]
Instead of estimating the additional Jacobi term separately, we combine it with $\varphi$ at the critical points using the adjacent-degree identity \cite[equation (3.133)]{STW}
\[
 (2k+\delta)(1-x^2)\frac{d}{dx}P_k^{(0,\delta)}(x)
 =-k\bigl(\delta+(2k+\delta)x\bigr)P_k^{(0,\delta)}(x)
   +2k(k+\delta)P_{k-1}^{(0,\delta)}(x).
\]
Setting $x=1-2u_i$ and using $(P_k^{(0,\delta)})'(1-2u_i)=0$, we obtain
\[
 \varphi(u_i)+u_iP_{k-1}^{(0,\delta)}(1-2u_i)
 =\Bigl(1+u_i-\frac{v}{k+\delta}u_i^2\Bigr)\varphi(u_i).
\]
The multiplier is positive, since
\begin{equation}\label{eq:theta-three-factor}
 1+u-\frac{v}{k+\delta}u^2
 =(1-u)(1+2u)+\frac{\delta}{k+\delta}u^2
 \ge(1-u)(1+2u)>0\qquad(0<u<1).
\end{equation}
Therefore, to preserve the sign of $\varphi(u_i)$ in \eqref{eq:theta-three-transform}, it suffices to prove
\[
 |g(u_i)|<(1-u_i)(1+2u_i)|\varphi(u_i)|
 \qquad(1\le i\le k-1).
\]
To account for the factor $1-u_i$, we apply \eqref{eq:theta-square-bound} to $g(u)/(1-u)$. This polynomial has degree at most $k-1$ and vanishes at both endpoints because $k\ge5$. We deduce
\begin{equation}\label{eq:theta-three-bound}
 \Bigl|\frac{g(u_i)}{(1-u_i)(1+2u_i)\varphi(u_i)}\Bigr|^2
 \le\frac{k(k+\delta+1)u_i(1-u_i)}{(1+2u_i)^2}\,I(v)
 \le\frac{v(v+2)}{48}\,I(v),
\end{equation}
where
\[
 I(v):=\int_0^1u(1-u)^{v-9} \Bigl(6-(4v-7)u+\frac{v^2-3v+4}{2}u^2\Bigr)^2du.
\]
The second inequality in \eqref{eq:theta-three-bound} follows from \eqref{eq:jacobi-parameter-bound} and
\[
 (1+2u)^2-12u(1-u)=(1-4u)^2\ge0.
\]
Termwise integration, using
\[
 \int_0^1u^{j+1}(1-u)^{v-9}du
 =\frac{(j+1)!}{\prod_{i=0}^{j+1}(v-8+i)}
 \qquad(0\le j\le4),
\]
simplifies the integral to
\[
 I(v)=\frac{6T(v)}{\prod_{i=3}^{8}(v-i)},\qquad
 T(v)=v^4-22v^3+190v^2-739v+1100.
\]
Assume first that $v\ge12$. To bound $I(v)$, we compare $T(v)$ with a product of four factors from the denominator:
\[
(v-3)(v-4)(v-5)(v-7)-T(v)=(v-12)\bigl(3(v-4)^2+v+26\bigr)+208>0.
\]
It follows that
\begin{equation}\label{eq:theta-three-strict-bound}
 \frac{v(v+2)}{48}I(v)
 <\frac{v(v+2)}{8(v-6)(v-8)}<1,
\end{equation}
where the last inequality follows from
\[
 8(v-6)(v-8)-v(v+2)=(v-12)(7v-30)+24>0.
\]
By \eqref{eq:theta-three-bound} and \eqref{eq:theta-three-strict-bound}, the required bound on $g(u_i)$ holds. Together with the positive multiplier in \eqref{eq:theta-three-factor}, this bound preserves the Jacobi sign pattern \eqref{eq:jacobi-critical-signs}:
\[
 (-1)^i\widetilde H(u_i)>0\qquad(1\le i\le k-1).
\]
To compare the number of zeros with the degree, we read off the leading coefficients from \eqref{eq:theta-three}:
\[
 [t^k]h_{M(\Theta(3,a,b))}^*=\binom{k+\delta-1}{k},\qquad
 [t^{k-1}]h_{M(\Theta(3,a,a))}^*=2k-1>0.
\]
Thus the degree $d$ equals $k$ when $\delta>0$ and $k-1$ when $\delta=0$. The endpoint values are
\[
 \widetilde H(0)=1,\qquad
 \widetilde H(1)=(-1)^k\binom{k+\delta-1}{k}.
\]
The intermediate value theorem gives a zero of $\widetilde H$ in each of
\[
 (0,u_1),\ (u_1,u_2),\ \ldots,\ (u_{k-2},u_{k-1}).
\]
If $\delta>0$, then $(-1)^k\widetilde H(1)>0$, so there is also a zero in $(u_{k-1},1)$. Under $t=-u/(1-u)$, these correspond to $k-1$ distinct negative zeros of $h_{M(\Theta(3,a,b))}^*$ when $\delta=0$, and $k$ when $\delta>0$. These counts equal the respective degrees, so all zeros are simple and negative.

For the remaining values $v=10,11$, corresponding respectively to $(a,b)=(3,3),(3,4)$, \eqref{eq:theta-three} gives
\[
 h_{M(\Theta(3,3,3))}^*(t)=1+18t+63t^2+54t^3+9t^4,\quad h_{M(\Theta(3,3,4))}^*(t)=1+23t+105t^2+125t^3+35t^4+t^5.
\]
At $0,-1/10,-1/2,-2,-6$, both polynomials have signs $+,-,+,-,+$, respectively. Hence each has four distinct negative zeros. The quintic tends to $-\infty$ as $t\to-\infty$ and is positive at $-6$, so it has a fifth negative zero. These counts exhaust the respective degrees.
\end{proof}

Applying \eqref{eq:jacobi-transform} to the four-parallel-class formula gives the same perturbation form as for $c=1$, with coefficient $2$ in place of $1$.

\begin{theorem}\label{thm:four-classes}
If $M$ is a rank-two matroid with parallel-class sizes $(1,1,a,b)$, where $1\le a\le b$, then $h_M^*$ is real-rooted. Its zeros are simple and negative, except when $a=b=1$, in which case $h_M^*(t)=(1+t)^2$.
\end{theorem}
\begin{proof}

If $a=b=1$, then \eqref{eq:four-class-formula} gives $h_M^*(t)=(1+t)^2$. Otherwise, set $k=a+1\ge2$ and $\delta=b-a\ge0$. By \eqref{eq:four-class-formula} and \eqref{eq:jacobi-transform},
\[
 (1-u)^kh_M^*\!\Bigl(-\frac{u}{1-u}\Bigr)=\varphi(u)+g(u),\qquad g(u)=2u(1-u)^{k-1}.
\]
The degree and endpoint conditions of \autoref{lem:theta-square} hold for $g$. For the integral condition, direct evaluation yields
\[
 \frac{k(k+\delta+1)}4\int_0^1\frac{g(u)^2}{u(1-u)}(1-u)^\delta du
 =\frac{(a+1)(b+2)}{(a+b)(a+b+1)}<1,
\]
because the denominator minus the numerator is $(a-1)(a+b)+(b-1)(b+2)>0$. Hence \autoref{lem:theta-square} supplies $k$ distinct zeros of $\varphi+g$ in $(0,1)$. Under $t=-u/(1-u)$, these become $k$ distinct negative zeros of $h_M^*$. Since $d=a+1=k$ by \eqref{eq:four-class-formula}, these are all its zeros.
\end{proof}

\section{Failure of real-rootedness}\label{sec:finite-infinite}
This section proves \autoref{thm:real-rootedness}(2)--(3). Both proofs use the coefficient formula and normalization below to violate the Laguerre inequality. The finite example uses exact rational evaluation. For the infinite family, the left-hand side of the inequality has a strictly negative limit at a fixed real point.
\subsection{The coefficient formula and normalization}

Fix integers $a\ge4$ and $b\ge a+2$. The theta graph $\Theta(4,a,b)$ has $m=a+b+4$ edges. Set $d=a+3\ge7$ and $\delta=b-a\ge2$. Then $m=2d+\delta-2$. To study this family as $d\to\infty$ with $\delta$ fixed, write
\[
 h_{d,\delta}^*(t):=h_{M(\Theta(4,d-3,d+\delta-3))}^*(t).
\]
The dual of $M(\Theta(4,a,b))$ is a connected rank-two matroid with parallel-class sizes $4,a,b$. By \eqref{eq:duality} and \eqref{circle:eq:fjs},
\begin{equation}\label{eq:theta-coefficients}
 h_{d,\delta}^*(t)=\sum_{j\ge0}\binom{m}{2j}t^j
 -p^*_{4,m}(t)-p^*_{a,m}(t)-p^*_{b,m}(t).
\end{equation}

The correction $p^*_{4,m}$ leaves the coefficients of $t^j$ unchanged for $j>4$. We isolate it and use \eqref{eq:block-identity} and \eqref{eq:correction-increment} to express the remaining terms with nonnegative coefficients. Define $S_{d,\delta}(t):=\sum_{j\ge0}s_jt^j$, where
\begin{equation}\label{eq:positive-coefficients}
 s_j:=\binom d j\binom{d+\delta-2}j
     +\binom{d-1}j\binom{d+\delta-2}{j-1}
     +\binom{d-2}j\binom{d+\delta-1}{j-1}.
\end{equation}
\begin{lemma}\label{lem:coeff}
The polynomial $h_{d,\delta}^*$ satisfies
\begin{equation}\label{eq:split}
 h_{d,\delta}^*(t)=S_{d,\delta}(t)-p^*_{4,m}(t).
\end{equation}
Moreover, $h_{d,\delta}^*$ has degree $d$, with leading coefficient $\binom{d+\delta-2}{d}$.
\end{lemma}
\begin{proof}

Applying \eqref{eq:block-identity} with $r=a+2$ and $m=a+b+4$ gives
\[
 \sum_{j\ge0}\binom m{2j}t^j-p^*_{a,m}(t)-p^*_{b,m}(t)
 =J_{a+3,b+1}(t)+p^*_{a+2,m}(t)-p^*_{a,m}(t).
\]
The last difference is the sum of the two consecutive increments in \eqref{eq:correction-increment} at $r=a+2$ and $r=a+1$. Adding these increments to $J_{a+3,b+1}$ and substituting $a=d-3$ and $b=d+\delta-3$ gives $S_{d,\delta}(t)$ by \eqref{eq:positive-coefficients}. Subtracting $p^*_{4,m}(t)$ now proves \eqref{eq:split} by \eqref{eq:theta-coefficients}.

Since $\delta\ge2$, the term $J_{a+3,b+1}=J_{d,d+\delta-2}$ has degree $d$ and leading coefficient $\binom{d+\delta-2}{d}$. The two increments have degree at most $d-1$, and $p^*_{4,m}$ has degree at most $4<d$. Therefore these lower-degree terms leave the leading term unchanged.

\end{proof}

To rule out real-rootedness without locating zeros, we use the classical Laguerre inequality~\cite[equation (1.2)]{CsordasEscassut}: Every nonzero real-rooted polynomial $p\in\mathbb R[x]$ satisfies
\begin{equation}\label{eq:laguerre}
 p'(x)^2-p(x)p''(x)\ge0\qquad(x\in\mathbb R).
\end{equation}

To study the limit as $d\to\infty$, we reverse $h_{d,\delta}^*$ and divide by its leading coefficient $l_{d,\delta}:=\binom{d+\delta-2}{d}$. At each fixed index of the reversed polynomial, the contribution of $p^*_{4,m}$ vanishes for all sufficiently large $d$. Since $d-1>4$, \eqref{eq:split} and \eqref{eq:positive-coefficients} give
\[
 \frac{[t^{d-1}]h_{d,\delta}^*(t)}{l_{d,\delta}}
 =\frac{d\binom{d+\delta-2}{d-1}+\binom{d+\delta-2}{d-2}}{\binom{d+\delta-2}{d}}
 =\frac{d((\delta+1)d-1)}{\delta(\delta-1)}.
\]
For fixed $\delta$, this quadratic growth suggests scaling the reversed variable by $d^{-2}$. Define
\begin{equation}\label{eq:normalized-reversal}
 Q_{d,\delta}(z):=l_{d,\delta}^{-1}\sum_{k=0}^d
 (-1)^k[t^{d-k}]h_{d,\delta}^*(t)(z/d^2)^k.
\end{equation}
By normalization, $Q_{d,\delta}(0)=1$. For $z\ne0$, reindexing by $j=d-k$ yields
\[
 Q_{d,\delta}(z)=l_{d,\delta}^{-1}
 \Bigl(-\frac z{d^2}\Bigr)^d
 h_{d,\delta}^*\Bigl(-\frac{d^2}{z}\Bigr).
\]
By \eqref{circle:eq:fjs}, $h_{d,\delta}^*(0)=1$, and \autoref{lem:coeff} gives $\deg h_{d,\delta}^*=d$. Thus both polynomials have degree $d$ and nonzero constant term. Their roots correspond, with multiplicity, under $\rho\mapsto-d^2/\rho$, which preserves real and nonreal zeros. Therefore,
\begin{equation}\label{A=B}
Q_{d,\delta}\text{ is real-rooted} \quad\Longleftrightarrow\quad h_{d,\delta}^*\text{ is real-rooted}.
\end{equation}
By \eqref{eq:laguerre} and \eqref{A=B}, to prove that $h_{d,\delta}^*$ is not real-rooted, it suffices to find a real $x$ such that
\[
 Q_{d,\delta}'(x)^2-Q_{d,\delta}(x)Q_{d,\delta}''(x)<0.
\]

\subsection{\texorpdfstring{Counterexamples}{Counterexamples}}
For the finite counterexample, it is enough to verify a violation of the Laguerre inequality at a single integer point.

\begin{proof}[Proof of \autoref{thm:real-rootedness}(2)]

For $\Theta(4,561,600)$, we have $m=1165$, $d=564$, and $\delta=39$. We verify that $Q_{564,39}$ violates \eqref{eq:laguerre} at the integer point $x=481$. Write $h_{564,39}^*(t)=\sum_{j=0}^{564}h_j^*t^j$. The coefficients $h_j^*$ are specified by \eqref{eq:split}, \eqref{eq:positive-coefficients}, and \eqref{circle:eq:fjs}, and $l_{564,39}=\binom{601}{564}$.

To evaluate $Q_{564,39}$ and its first two derivatives exactly, consider the polynomials $R_k(y)$ determined by $R_{565}(y)=0$ and
\[
 R_k(y)=(-1)^kh_{564-k}^*+yR_{k+1}(y),\qquad k=564,563,\ldots,0.
\]
Differentiating with respect to $y$ gives
\[
 R_k'(y)=R_{k+1}(y)+yR_{k+1}'(y),\qquad
 R_k''(y)=2R_{k+1}'(y)+yR_{k+1}''(y).
\]
By \eqref{eq:normalized-reversal} and the chain rule,
\begin{equation}\label{eq:finite-evaluation}
 Q_{564,39}^{(r)}(x)=\frac{R_0^{(r)}(x/564^2)}{l_{564,39}564^{2r}},\qquad r=0,1,2.
\end{equation}
We evaluate the three recurrences at $y=x/564^2=481/318096$ in exact rational arithmetic, starting from $R_{565}=R_{565}'=R_{565}''=0$. Equation~\eqref{eq:finite-evaluation} then gives
\begin{equation}\label{eq:finitebox}
 Q_{564,39}(x)>\frac{3}{10^{12}},\qquad
 |Q_{564,39}'(x)|<\frac{5}{10^{12}},\qquad
 Q_{564,39}''(x)>\frac{1}{10^{11}}.
\end{equation}
By \eqref{eq:finitebox},
\[
 Q_{564,39}'(x)^2-Q_{564,39}(x)Q_{564,39}''(x)
 <\frac{25-30}{10^{24}}<0.
\]
Therefore \eqref{eq:laguerre} and \eqref{A=B} imply that $h_{M(\Theta(4,561,600))}^*$ is not real-rooted.

\end{proof}

We next obtain the infinite family in \autoref{thm:real-rootedness}(3) by letting $d=a+3\to\infty$ with $\delta=29$, so $b=a+29$. To show that the Laguerre expression has a strictly negative limit at a fixed real point, we first determine the limits of $Q_{d,29}$ and its first two derivatives. The three binomial products in \eqref{eq:positive-coefficients} yield the factor $\alpha_k:=3k^2+85k+812$, $k\ge0$, in the limits below.
\begin{lemma}\label{lem:limit}
For $r\in\{0,1,2\}$, the following convergence is uniform on every bounded real interval:
\begin{equation}\label{eq:real-derivative-limits}
 \lim_{d\to\infty}Q_{d,29}^{(r)}(x)
 =(-1)^r27!\sum_{k\ge0}
 \frac{(-x)^k\alpha_{k+r}}{k!(k+r+29)!}.
\end{equation}
\end{lemma}
\begin{proof}
Fix $R>0$ and $r\in\{0,1,2\}$, and let $|x|\le R$. Since $d\to\infty$, we may assume $d\ge29$. Throughout the proof, $\delta=29$ and $m=2d+27$. Write $h_{d,29}^*(t)=\sum_{j=0}^d h_j^*t^j$, with the dependence of $h_j^*$ on $d$ understood. Stanley's nonnegativity theorem and \eqref{eq:split} give
\[
 0\le h_j^*\le s_j\qquad(0\le j\le d),\qquad
 h_j^*=s_j\qquad(4<j\le d),
\]
because $p^*_{4,m}$ has nonnegative coefficients and degree at most four. Thus the coefficients $s_j$ provide both a common bound and the eventual values needed for the coefficient limits.

Substitute $j=d-k$ into \eqref{eq:positive-coefficients} and use $\binom uj=\binom u{u-j}$ to obtain
\[
s_{d-k}=\binom d k\binom{d+27}{k+27}+\binom{d-1}{k-1}\binom{d+27}{k+28}+\binom{d-2}{k-2}\binom{d+28}{k+29}\qquad(0\le k\le d).
\]
For fixed integers $\kappa$ and $j\ge0$,
\[
 \frac{\binom{d+\kappa}{j}}{d^j}=\frac1{j!}\prod_{i=0}^{j-1}\Big(1+\frac{\kappa-i}{d}\Big)\xrightarrow[d\to\infty]{}\frac1{j!}.
\]
For each fixed $k$, summing the limits of the nonzero products yields
\[
 d^{-27-2k}s_{d-k}\xrightarrow[d\to\infty]{}
 \frac{(k+29)(k+28)+k(k+29)+k(k-1)}{k!(k+29)!}
 =\frac{\alpha_k}{k!(k+29)!}.
\]
Here the second binomial product is zero for $k=0$, and the third is zero for $k=0,1$. Since
\[
\frac{l_{d,29}}{d^{27}}=\frac1{27!}\prod_{i=1}^{27}\Big(1+\frac{i}{d}\Big)\xrightarrow[d\to\infty]{}\frac1{27!},
\]
the normalized coefficients satisfy, for each fixed $k$,
\begin{equation}\label{eq:normalized-coefficient-limit}
 \frac{s_{d-k}}{l_{d,29}d^{2k}}=\frac{d^{-27-2k}s_{d-k}}{d^{-27}l_{d,29}}\xrightarrow[d\to\infty]{}\frac{27!\alpha_k}{k!(k+29)!}.
\end{equation}
Coefficientwise convergence alone does not justify taking limits in the differentiated sums. Therefore, we need a summable bound independent of $d$. Since $d\ge29$, we have $d+28\le2d$. Combining this with $l_{d,29}\ge d^{27}/27!$ and $\binom ij\le i^j/j!$ bounds the three binomial products above and gives, for $0\le k\le d$,
\begin{equation}\label{eq:normalized-coefficient-bound}
 0\le\frac{s_{d-k}}{l_{d,29}d^{2k}}\le\frac{27!\,2^{k+29}\alpha_k}{k!(k+29)!}.
\end{equation}
After differentiating \eqref{eq:normalized-reversal} $r$ times, we reindex with $k=j-r$:
\[
 Q_{d,29}^{(r)}(x)=\sum_{j=r}^d(-1)^j\frac{j!}{(j-r)!} \frac{h_{d-j}^*x^{j-r}}{l_{d,29}d^{2j}}=\sum_{k=0}^{d-r}(-1)^{k+r}\frac{(k+r)!}{k!}\frac{h_{d-k-r}^*x^k}{l_{d,29}d^{2(k+r)}}.
\]
Applying \eqref{eq:normalized-coefficient-bound} at index $k+r$ and using $0\le h_{d-k-r}^*\le s_{d-k-r}$ and $|x|^k\le R^k$, we obtain the uniform bound
\begin{equation}\label{eq:derivative-summable-bound}
 \frac{(k+r)!}{k!}\frac{h_{d-k-r}^*|x|^k}{l_{d,29}d^{2(k+r)}}
 \le\frac{27!\,2^{r+29}\alpha_{k+r}(2R)^k}{k!(k+r+29)!},\qquad 0\le k\le d-r.
\end{equation}
The series obtained by summing the right-hand side of \eqref{eq:derivative-summable-bound} over $k\ge0$ has consecutive-term ratio
\[
 \frac{\alpha_{k+r+1}}{\alpha_{k+r}}
 \frac{2R}{(k+1)(k+r+30)}
 \le\frac{4R}{(k+1)^2}\xrightarrow[k\to\infty]{}0,
\]
since $2\alpha_k-\alpha_{k+1}=3k^2+79k+724>0$ for $k\ge0$. Thus the bounding series converges.

Extend the differentiated summand by zero for $k>d-r$. For each fixed $k$, we have $h_{d-k-r}^*=s_{d-k-r}$ once $d>k+r+4$. Applying \eqref{eq:normalized-coefficient-limit} at index $k+r$ therefore gives uniform convergence of every fixed finite initial sum on $[-R,R]$. The summable bound \eqref{eq:derivative-summable-bound} controls the tails of both the differentiated polynomials and the limiting series, uniformly in $d$ and $x$. Choosing a common truncation with uniformly small tails and then taking $d$ sufficiently large proves \eqref{eq:real-derivative-limits} uniformly on $[-R,R]$. Since $R>0$ was arbitrary, this proves the lemma.
\end{proof}

By \autoref{lem:limit}, a strictly negative value of the limiting Laguerre expression at one real point persists for all sufficiently large $d$.

\begin{proof}[Proof of \autoref{thm:real-rootedness}(3)]

Take $x=290$. We bound the three limits in \eqref{eq:real-derivative-limits} using their partial sums through $k=32$. For each $r\in\{0,1,2\}$, the series alternates in sign. As shown in the proof of \autoref{lem:limit}, $\alpha_{j+1}\le2\alpha_j$ for $j\ge0$. Thus the ratio of the absolute values of consecutive terms satisfies
\[
 \frac{x\alpha_{k+r+1}}{(k+1)(k+r+30)\alpha_{k+r}}
 \le\frac{2x}{(k+1)(k+r+30)}
 \le\frac{580}{33\cdot62}<1,\qquad k\ge32.
\]
Since the tail terms decrease in absolute value to zero, \eqref{eq:real-derivative-limits} and the alternating-series remainder bound give
\[
 \Bigl|
 \lim_{d\to\infty}Q_{d,29}^{(r)}(x)
 -(-1)^r27!\sum_{k=0}^{32}
 \frac{(-x)^k\alpha_{k+r}}{k!(k+r+29)!}
 \Bigr|
 \le\frac{27!x^{33}\alpha_{33+r}}{33!(62+r)!},
 \qquad r=0,1,2.
\]
Starting with $(-1)^r27!\alpha_r/(r+29)!$, we evaluate successive terms by multiplying the $k$th term by $-\frac{x\alpha_{k+r+1}}{(k+1)(k+r+30)\alpha_{k+r}}$. Exact rational summation through $k=32$, combined with the preceding remainder bound, yields
\[
 \lim_{d\to\infty}Q_{d,29}(x)>\frac{60}{10^{10}},\qquad
 \Bigl|\lim_{d\to\infty}Q_{d,29}'(x)\Bigr|<\frac{21}{10^{10}},\qquad
 \lim_{d\to\infty}Q_{d,29}''(x)>\frac8{10^{10}}.
\]
Since all three inequalities are strict, the same bounds hold simultaneously for $Q_{d,29}(x)$, $Q_{d,29}'(x)$, and $Q_{d,29}''(x)$ for all sufficiently large $d$. Consequently,
\[
 Q_{d,29}'(x)^2-Q_{d,29}(x)Q_{d,29}''(x) <\frac{21^2-60\times8}{10^{20}}=-\frac{39}{10^{20}}<0.
\]
By \eqref{eq:laguerre} and \eqref{A=B}, $h_{d,29}^*$ is not real-rooted for all sufficiently large $d=a+3$.

\end{proof}

\section{Ultra log-concavity in rank two and corank two}\label{sec:rank-two}
We prove \autoref{thm:rank-two-full-unimodal} by bounding consecutive coefficient ratios in both directions. A nonnegative expansion gives the degree, positivity, and the upper bound. For the lower bound, we use discrete convexity of the positive part of a suitable combination of correction terms. The first internal index requires a separate count of bases. Disconnected matroids and duality are treated at the end.

\subsection{A nonnegative expansion and the degree}\label{sec:ulc-expansion}
Until the proof of \autoref{thm:rank-two-full-unimodal}, let $M$ be a connected rank-two matroid on $m$ elements, with parallel-class sizes $a_1,\ldots,a_s$, where $s\ge3$ and $\sum_{i=1}^s a_i=m$. Write
\[
 h_M^*(t)=\sum_{j\ge0}h_j^*t^j.
\]

The coefficient formula \eqref{circle:eq:fjs} is a difference of nonnegative sums. For $j\ge1$, multiplication by $j$ and the identity $j\binom kj=k\binom{k-1}{j-1}$ give a binomial product invariant under $k\leftrightarrow m-k$. We therefore pair the corresponding terms to obtain nonnegative weights. To count the correction terms at index $k$, set
\[
 c_k:=\#\{i\in\{1,\ldots,s\}:a_i\ge k\}\qquad(k\ge1).
\]
For \(1\le k\le \lfloor m/2\rfloor\), define
\[
 w_k:=
 \begin{cases}
 m-kc_k-(m-k)c_{m-k},&2k<m,\\
 k(1-c_k),&2k=m.
 \end{cases}
\]

\begin{proposition}\label{prop:rank-two-coefficient-expansion}
For every $j\ge1$,
\begin{equation}\label{eq:rank-two-coefficient-expansion}
 jh_j^*=\sum_{k=1}^{\lfloor m/2\rfloor} w_k\binom{k-1}{j-1}\binom{m-k-1}{j-1},
 \qquad w_k\ge0.
\end{equation}
If $d=\deg h_M^*\ge1$, then
\begin{equation}\label{eq:rank-two-tail-ratio}
 j(j+1)h_{j+1}^*\le(d-j)(m-d-j)h_j^*\qquad(1\le j<d).
\end{equation}
\end{proposition}
\begin{proof}
With $a=m-1$, \eqref{eq:majority} reads
\[
 \binom m{2j}=C_j(m-1,m)
 =\sum_{k=1}^{m-1}\binom kj\binom{m-k-1}{j-1}.
\]
Summing the defining formula for $C_j(a_i,m)$ over $i$ and then interchanging the $i$- and $k$-sums yields
\[
 \sum_{i=1}^{s}C_j(a_i,m)
 =\sum_{i=1}^{s}\sum_{k=1}^{a_i}\binom kj\binom{m-k-1}{j-1}
 =\sum_{k=1}^{m-1}c_k\binom kj\binom{m-k-1}{j-1}.
\]
Taking the coefficient of $t^j$ in \eqref{circle:eq:fjs}, substituting the two identities above, and multiplying by $j$ gives
\[
 jh_j^*=\sum_{k=1}^{m-1}k(1-c_k)
       \binom{k-1}{j-1}\binom{m-k-1}{j-1}.
\]
For $2k<m$, the terms indexed by $k$ and $m-k$ have the same binomial product, and
\[
 k(1-c_k)+(m-k)(1-c_{m-k})=w_k.
\]
When $2k=m$, the middle term has coefficient $w_k$. This proves the identity in \eqref{eq:rank-two-coefficient-expansion}.

We next show that $w_k\ge0$. Fix $1\le k\le\lfloor m/2\rfloor$. If $c_{m-k}=0$, then $kc_k\le\sum_{i=1}^{s}a_i=m$. Thus $w_k=m-kc_k\ge0$ when $2k<m$, and $w_k=k>0$ when $2k=m$. Otherwise, relabel the parallel classes so that $a_s\ge m-k$. Since
\[
 \sum_{i=1}^{s-1}a_i=m-a_s\le k,\qquad
 a_i\le k-(s-2)<k\quad(i<s),
\]
we have $c_k=c_{m-k}=1$. Hence $w_k=0$.

To prove \eqref{eq:rank-two-tail-ratio}, we first identify the largest index with positive weight. Note that
\[
 \binom{k-1}{j-1}\binom{m-k-1}{j-1}>0
 \quad\Longleftrightarrow\quad 1\le j\le k.
\]
Since no cancellation occurs, \eqref{eq:rank-two-coefficient-expansion} implies $d=\max\{k:w_k>0\}\le\lfloor m/2\rfloor$ whenever $d\ge1$. Fix $1\le j<d$. For $j\le k\le d$, we have
\[
 (d-j)(m-d-j)-(k-j)(m-k-j)=(d-k)(m-d-k)\ge0.
\]
Using this comparison in \eqref{eq:rank-two-coefficient-expansion} at $j+1$ and $j$ gives
 \begin{align*}
 j(j+1)h_{j+1}^*
 &=\frac1j\sum_{k=j}^{d}w_k(k-j)(m-k-j)
   \binom{k-1}{j-1}\binom{m-k-1}{j-1}\\
 &\le\frac{(d-j)(m-d-j)}j
   \sum_{k=j}^{d}w_k\binom{k-1}{j-1}\binom{m-k-1}{j-1}\\
 &=(d-j)(m-d-j)h_j^*.
 \end{align*}
This completes the proof.
\end{proof}

Put $a_{\max}=\max_{1\le i\le s}a_i$ and $d=\deg h_M^*$. The degree formula below also follows from~\cite[Section 3 and Proposition 4.2]{HigashitaniMatsushita}. The nonnegative expansion gives the same formula and also verifies the positivity needed for the subsequent coefficient comparisons.

\begin{lemma}\label{lem:rank-two-degree-positive}
Suppose $m\ge4$. Then
\[
 d=\min\{\lfloor m/2\rfloor,m-a_{\max}-1\},
 \qquad h_0^*,\ldots,h_d^*>0.
\]
\end{lemma}
\begin{proof}
We locate the last positive weight in \eqref{eq:rank-two-coefficient-expansion}. Since $s\ge3$, we have $a_{\max}\le m-2$. Suppose $a_{\max}\ge m/2$. Every index $m-a_{\max}-1<k\le\lfloor m/2\rfloor$ satisfies $m-k\le a_{\max}$, so $c_{m-k}\ge1$ and $w_k=0$ by the proof of \autoref{prop:rank-two-coefficient-expansion}. Relabel the classes so that $a_s=a_{\max}$ and take $k=m-a_{\max}-1$. Then $1\le k<m/2$ and $c_{m-k}=0$, since $m-k=a_{\max}+1$. The other classes have total size $k+1$, so

\[
 kc_k=k+k\,\#\{i<s:a_i\ge k\}\le k+\sum_{i=1}^{s-1}a_i=2k+1.
\]

Consequently,

\[
 w_k=m-kc_k\ge m-2k-1=2a_{\max}-m+1\ge1.
\]

If $a_{\max}<m/2$, take $k=\lfloor m/2\rfloor$. For even $m$, we have $c_k=0$ and $w_k=m/2>0$. For odd $m$, at most two parallel classes have size $k$, and $w_k=m-kc_k\ge1$. In either case,
\[
 \max\{k:w_k>0\}=\min\{\lfloor m/2\rfloor,m-a_{\max}-1\}.
\]
By \autoref{prop:rank-two-coefficient-expansion}, this maximum equals $d$. The $k=d$ term in \eqref{eq:rank-two-coefficient-expansion} is positive for every $1\le j\le d$, while all other terms are nonnegative. Thus $h_j^*>0$ for $1\le j\le d$. The constant term is $h_0^*=1$ by \eqref{circle:eq:fjs}.
\end{proof}

\subsection{A lower coefficient bound}

To obtain the lower coefficient bound, we estimate the total correction in \eqref{circle:eq:fjs} from above. The next lemma ensures positivity at the endpoints used in the convexity argument.

\begin{lemma}\label{lem:correction-normalized-ratio}
Let \(m,v,j\) be integers with \(m\ge2v+1\) and \(1\le j<v\). Then
\[
 (j+1)^2C_{j+1}(v,m)>(v-j)(m-v-j-1)C_j(v,m).
\]
\end{lemma}
\begin{proof}
To compare the two sides, we let the first argument of the correction terms increase from $j$ to $v$. For $j\le u\le v$, set
\[
 \xi_u=(j+1)^2C_{j+1}(u,m)-(u-j)(m-u-j-1)C_j(u,m),
 \quad \omega_u=\binom uj\binom{m-u-1}{j-1}.
\]
We have $\xi_j=0$. Extracting the coefficients of $t^j$ and $t^{j+1}$ in \eqref{eq:correction-increment} with $r=u$ gives, for $j<u\le v$,
\[
 \xi_u-\xi_{u-1}
 =\frac{(u-j)(m-u)}j\omega_u-(m-2u)C_j(u-1,m).
\]
To write both terms over the same index set, use the binomial identity
\[
 (k+1-j)(m-k-1)\omega_{k+1}=(k+1)(m-k-j)\omega_k
 \qquad(j\le k<u).
\]
Subtracting $(k-j)(m-k)\omega_k$ from both sides and summing over $k=j,\ldots,u-1$ gives
\[
 (u-j)(m-u)\omega_u
 =\sum_{k=j}^{u-1}\bigl((j+1)m-(2j+1)k-j\bigr)\omega_k,
\]
because the lower boundary term vanishes. Substituting this identity and $C_j(u-1,m)=\sum_{k=j}^{u-1}\omega_k$ into the expression for $\xi_u-\xi_{u-1}$ yields
\[
 \xi_u-\xi_{u-1}
 =\frac1j\sum_{k=j}^{u-1}\bigl(m-k+j(2u-2k-1)\bigr)\omega_k>0.
\]
For $j\le k<u\le v$, the assumptions give $m-k>0$, $2u-2k-1\ge1$, and $\omega_k>0$. Thus $\xi_v>\xi_j=0$, as required.
\end{proof}

A real-valued function $f$ on an interval of integers is called \emph{discretely convex} if
\[
 f(k-1)+f(k+1)\ge2f(k)
\]
whenever $k-1,k,k+1$ belong to its domain. Equivalently, its successive differences $f(k+1)-f(k)$ are nondecreasing.

Fix integers $m,j$ with $1\le j<\lfloor m/2\rfloor$ and a real number $\lambda\ge0$. For integer $a$, define
\[
 F(a):=(j+1)^2C_{j+1}(a,m)-\lambda C_j(a,m)
 \qquad(0\le a<m).
 \]
The sum $\sum_iF(a_i)$ is the total correction in the comparison of $(j+1)^2h_{j+1}^*$ with $\lambda h_j^*$. Since $F$ need not be discretely convex, we seek a convex upper bound by taking its positive part on $0\le a\le\lfloor m/2\rfloor$:
\[
 F_+(a):=\max\{F(a),0\}.
\]
\begin{lemma}\label{lem:correction-positive-part}
The function $F_+$ is nondecreasing and discretely convex, with $F_+(0)=F_+(1)=0$. Moreover,
\begin{equation}\label{eq:convex-secant}
 (v-1)F_+(a)\le(a-1)F_+(v)
 \qquad(1\le a\le v\le\lfloor m/2\rfloor),
\end{equation}
where $a,v$ are integers.
\end{lemma}
\begin{proof}
We prove monotonicity and discrete convexity of $F_+$ by examining the increments of $F$. By definition, $F(0)=0$ and $F(1)\le0$, so $F_+(0)=F_+(1)=0$. For \(j\le a\le \lfloor m/2\rfloor\), put \(\omega_a=C_j(a,m)-C_j(a-1,m)=\binom aj\binom{m-a-1}{j-1}\). Then
\[
 F(a)-F(a-1)
   =\omega_a\Bigl(\frac{(j+1)(a-j)(m-a-j)}j-\lambda\Bigr).
\]
The increment vanishes for $1\le a<j$. For $j\le a<\lfloor m/2\rfloor$, we have $m-a-1\ge a+1$, so
\[
 \frac{\omega_{a+1}}{\omega_a}
 =\frac{a+1}{a+1-j}\Bigl(1-\frac{j-1}{m-a-1}\Bigr)
 \ge\frac{a+2-j}{a+1-j}>1.
\]
Also,
\[
 (a+1-j)(m-a-j-1)-(a-j)(m-a-j)=m-2a-1>0.
\]
Therefore, both factors in $F(a)-F(a-1)$ increase with $a$. Since $\omega_a>0$, a positive increment is followed only by positive, nondecreasing increments.

Fix $1\le a<\lfloor m/2\rfloor$ with $F(a)>0$. The identity $F(a)=\sum_{u=1}^a(F(u)-F(u-1))$ shows that a positive increment occurs at or before $a$. Hence
\[
 F(a+1)-F(a)\ge F(a)-F(a-1)>0.
\]
Consequently, $F_+(a+1)>F_+(a)$, and $F_+\ge F$ gives
\[
 F_+(a-1)+F_+(a+1)\ge F(a-1)+F(a+1)\ge2F(a)=2F_+(a).
\]
If $F(a)\le0$, then $F_+(a)=0$, so monotonicity and discrete convexity at $a$ follow from $F_+\ge0$. Thus both properties hold at every interior index. The equality $F_+(0)=F_+(1)=0$ also gives the monotonicity step from $0$ to $1$.

For \eqref{eq:convex-secant}, both sides vanish when $a=1$, including $v=1$. For $1<a\le v$, we compare the averages of the nondecreasing increments, using $F_+(1)=0$:
\[
 \frac{F_+(a)}{a-1}
 =\frac1{a-1}\sum_{u=2}^a\bigl(F_+(u)-F_+(u-1)\bigr)
 \le\frac1{v-1}\sum_{u=2}^v\bigl(F_+(u)-F_+(u-1)\bigr)
 =\frac{F_+(v)}{v-1}.
\]
Multiplication by $(a-1)(v-1)$ proves \eqref{eq:convex-secant}.
\end{proof}

The two lemmas now give the lower bound that complements \eqref{eq:rank-two-tail-ratio}.

\begin{proposition}\label{prop:rank-two-lower-ratio}
Suppose $m\ge4$. For $1\le j<d$,
\begin{equation}\label{eq:rank-two-lower-ratio}
(j+1)^2h_{j+1}^*\ge(d-j)(m-d-j-1)h_j^*.
\end{equation}
\end{proposition}
\begin{proof}
By \autoref{lem:rank-two-degree-positive}, $m\ge2d$. Fix $1\le j<d$ and choose $\lambda=(d-j)(m-d-j-1)\ge0$ in the definitions of $F$ and $F_+$ preceding \autoref{lem:correction-positive-part}. Set
\begin{equation}\label{eq:rank-two-sigma}
 \sigma=(j+1)^2\binom m{2j+2}-\lambda\binom m{2j}
 =\Bigl[\frac{(j+1)(m-2j)(m-2j-1)}{2(2j+1)}-\lambda\Bigr]\binom m{2j}.
\end{equation}
It follows from \eqref{circle:eq:fjs} that
\[
 (j+1)^2h_{j+1}^*-\lambda h_j^*=\sigma-\sum_{i=1}^sF(a_i).
\]
Thus \eqref{eq:rank-two-lower-ratio} is equivalent to
\[
\sum_{i=1}^sF(a_i)\le\sigma.
\]
The degree formula in \autoref{lem:rank-two-degree-positive} determines the endpoint for the convexity bound. If $a_{\max}\ge m/2$, we first separate the largest class. Otherwise, we distinguish odd and even $m$.

\emph{A parallel class of size at least \(m/2\).} Here $d=m-a_{\max}-1\ge2$. The assumption $a_{\max}\ge m/2$ implies $a_{\max}\ge d+1$, so $m=a_{\max}+d+1\ge2d+2$. Relabel the parallel classes so that $a_s=a_{\max}$. We first express $\sigma$ in terms of $F(a_{\max})$ and $F(d)$. Equation~\eqref{eq:overlap-identity} with $r=a_{\max}$ and $m-a_{\max}-1=d$ gives
\[
 \binom m{2j}-C_j(a_{\max},m)-C_j(d,m)
 =\binom{a_{\max}}j\binom dj.
\]
Since $\lambda=(d-j)(a_{\max}-j)$, substituting this identity at $j$ and $j+1$ into the definitions of $\sigma$ and $F$ yields
\[
 \sigma-F(a_{\max})-F(d)=(j+1)^2\binom{a_{\max}}{j+1}\binom d{j+1}
                 -(d-j)(a_{\max}-j)\binom{a_{\max}}j\binom dj=0.
\]
Therefore, it remains to show $\sum_{i=1}^{s-1}F(a_i)\le F(d)$. There are at least two remaining parallel classes, with total size $d+1$. Hence $a_i\le d$ for $i<s$ and $\sum_{i=1}^{s-1}(a_i-1)\le d-1$. By \autoref{lem:correction-normalized-ratio} with $v=d$, we have $F_+(d)=F(d)>0$. Applying \eqref{eq:convex-secant} with $v=d$ gives
\[
 \sum_{i=1}^{s-1}F(a_i)\le\sum_{i=1}^{s-1}F_+(a_i)
 \le\frac{\sum_{i=1}^{s-1}(a_i-1)}{d-1}F_+(d)\le F(d).
\]

\emph{Odd $m$, with $a_{\max}<m/2$.} Here $m=2d+1$ and $\lambda=(d-j)^2$, so \eqref{eq:rank-two-sigma} gives $\sigma=\binom m{2j}(d-j)(d+1)/(2j+1)>0$. Equation~\eqref{eq:overlap-identity} with $r=d$ reads $\binom m{2j}-2C_j(d,m)=\binom d j^2$. Using this identity at $j$ and $j+1$ in the definitions of $F$ and $\sigma$, we obtain
\[
 2F(d)-\sigma
 =\lambda\binom dj^2-(j+1)^2\binom d{j+1}^2=0.
\]
Therefore, $F_+(d)=F(d)=\sigma/2$. Since $a_i\le d$ for every $i$ and $s\ge3$, \eqref{eq:convex-secant} with $v=d$ gives
\[
 \sum_iF(a_i)\le\sum_iF_+(a_i)
 \le\frac{2d+1-s}{d-1}F_+(d)\le2F_+(d)=\sigma.
\]

\emph{Even $m$, with $a_{\max}<m/2$.} Here $m=2d$, $a_{\max}\le d-1$, and $\lambda=(d-j)(d-j-1)$. If $j=d-1$, then $\lambda=0$ and \eqref{eq:rank-two-lower-ratio} follows from $h_d^*>0$ in \autoref{lem:rank-two-degree-positive}. Assume $j\le d-2$, so $d\ge3$. We first express $F(d-1)$ in terms of $\sigma$ to bound the sum of $F(a_i)$. Equation~\eqref{eq:block-identity} with $r=d-1$ reads $2C_j(d-1,m)=\binom m{2j}-\binom d j^2$. Applying this identity at $j$ and $j+1$ to $F(d-1)$ gives
\[
 2F(d-1)-\sigma
 =\lambda\binom dj^2-(j+1)^2\binom d{j+1}^2
 =-(d-j)\binom dj^2.
\]
By \autoref{lem:correction-normalized-ratio} with $v=d-1$, we have $F_+(d-1)=F(d-1)>0$. To use $\sum_i a_i=2d$, we take the secant from $0$ rather than from $1$. Discrete convexity and $F_+(0)=0$, established in \autoref{lem:correction-positive-part}, imply the following bound by averaging the increments:
\[
 F_+(a)\le\frac a{d-1}F_+(d-1)\qquad(1\le a\le d-1).
\]
Summing over the parallel-class sizes and using $\sum_i a_i=2d$, we derive
\[
 \sum_iF(a_i)\le\sum_iF_+(a_i)\le\frac{2d}{d-1}F(d-1)
 =\frac d{d-1}\Bigl[\sigma-(d-j)\binom dj^2\Bigr].
\]
Equation~\eqref{eq:rank-two-sigma} with $m=2d$ and $\lambda=(d-j)(d-j-1)$ gives $\sigma=\frac{d(d-j)}{2j+1}\binom{2d}{2j}$. Therefore, it suffices to prove $\binom{2d}{2j}\le(2j+1)\binom dj^2$. For $j<k\le\min\{2j,d\}$,
\[
 \frac{\binom d k\binom d{2j-k}}{\binom d{k-1}\binom d{2j-k+1}}=\frac{d-k+1}{d-2j+k}\frac{2j-k+1}{k}<1,
\]
because both factors lie in $(0,1)$. Together with symmetry about $k=j$, this decrease shows that $\binom d k\binom d{2j-k}$ is largest at $k=j$. Counting $2j$-element subsets of two disjoint $d$-element sets by the number chosen from the first set gives
\[
 \binom{2d}{2j}
 =\sum_{k=0}^{2j}\binom d k\binom d{2j-k}
 \le(2j+1)\binom d j^2.
\]
Consequently,
\[
\sum_i F(a_i)\le\frac{d}{d-1}\Bigl[\sigma-(d-j)\binom dj^2\Bigr]\le \sigma.
\]
This completes the proof.
\end{proof}

\subsection{\texorpdfstring{Proof of \autoref{thm:rank-two-full-unimodal}}{Proof of Theorem 1.2}}\label{sec:rank-two-proof}
The ratio bounds cannot be combined at $j=1$, because \eqref{eq:rank-two-lower-ratio} has no $j=0$ case. We treat this index by combining \eqref{eq:rank-two-coefficient-expansion} with a count of bases.

\begin{lemma}\label{lem:rank-two-first-index}
If $d\ge2$, then
\[
                         2d h_2^*\le(d-1)(h_1^*)^2.
\]
\end{lemma}
\begin{proof}
By \eqref{eq:rank-two-coefficient-expansion}, $d\le\lfloor m/2\rfloor$, so $m\ge2d\ge4$. The degree formula in \autoref{lem:rank-two-degree-positive} now gives $a_{\max}\le m-d-1$. The $k=1$ term in \eqref{eq:rank-two-coefficient-expansion} contributes to $h_1^*$ but not to $h_2^*$. For $2\le k\le d$, the inequality $(k-1)(m-k-1)\le(d-1)(m-d-1)$ follows from $d\le m/2$. Keeping the $k=1$ term separate therefore gives
\begin{equation}\label{h12}
 h_1^*=\sum_{k=1}^d w_k,\quad 2h_2^*=\sum_{k=2}^d w_k(k-1)(m-k-1)\le(d-1)(m-d-1)(h_1^*-w_1).
\end{equation}
Since $h_1^*-w_1=\sum_{k=2}^d w_k\ge0$, it suffices to prove
\[
 h_1^*+w_1\ge d(m-d-1).
\]
Indeed, multiplying this bound by $(d-1)(h_1^*-w_1)$ and combining it with \eqref{h12} gives
\[
 2dh_2^*\le(d-1)(h_1^*+w_1)(h_1^*-w_1)
 \le(d-1)(h_1^*)^2.
\]
The bases are pairs from distinct parallel classes, so \eqref{circle:eq:fjs} and the definition of $w_1$ give
\[
 h_1^*=|\mathcal B(M)|-m=\frac{m^2-2m-\sum_i a_i^2}{2},
 \qquad w_1=m-s\ge0.
\]
If $a_{\max}<d$, then $\sum_i a_i^2\le m(d-1)$, so
\[
 h_1^*+w_1\ge h_1^*\ge\frac{m(m-d-1)}2\ge d(m-d-1).
\]
Suppose $a_{\max}\ge d$. There are $a_{\max}(m-a_{\max})$ bases meeting a largest parallel class. Choose one representative from each of the other $s-1$ classes. Pairing one fixed representative with each of the remaining $s-2$ representatives gives $s-2$ further bases. Since $h_1^*+w_1=|\mathcal B(M)|-s$, we obtain

\[
 \begin{aligned}
 h_1^*+w_1&\ge a_{\max}(m-a_{\max})-2\\
      &=d(m-d-1)+(a_{\max}-d)(m-d-1-a_{\max})+a_{\max}-2\\
      &\ge d(m-d-1),
 \end{aligned}
\]
where the last inequality uses $2\le d\le a_{\max}\le m-d-1$.
\end{proof}

We now combine the coefficient bounds and treat the disconnected case.
\begin{proof}[Proof of \autoref{thm:rank-two-full-unimodal}]
By \eqref{eq:duality}, it suffices to consider rank-two matroids. Deleting loops does not change $h_M^*$, so we may assume that $M$ is loopless. Suppose first that \(M\) is connected of rank two on \(m\ge4\) elements. By \autoref{lem:rank-two-degree-positive}, $d\le\lfloor m/2\rfloor$ and $h_0^*,\ldots,h_d^*>0$. For $2\le j<d$, the upper bound \eqref{eq:rank-two-tail-ratio} at $j$ and the lower bound \eqref{eq:rank-two-lower-ratio} at $j-1$ share the positive factor $m-d-j\ge d-j>0$. Combining these bounds gives
\[
 \frac{h_{j+1}^*}{h_j^*}
 \le\frac{(d-j)(m-d-j)}{j(j+1)}
 \le\frac{j(d-j)}{(j+1)(d-j+1)}
       \frac{h_j^*}{h_{j-1}^*}.
\]
Rearranging yields \eqref{eq:ulc} for $2\le j<d$. For $d\ge2$, \eqref{eq:ulc} at $j=1$ follows from \autoref{lem:rank-two-first-index}, since $h_0^*=1$. When $d\le1$, there is no internal index. The only remaining connected ground-set size is $m=3$, namely $M=U_{2,3}$, for which $h_M^*=1$ by \eqref{circle:eq:fjs}.

If \(M\) is disconnected, then \(M=U_{1,a_1}\oplus U_{1,a_2}\) for positive integers \(a_1,a_2\), as noted in Section~\ref{sec:coefficients}. By~\cite[Example 2.13]{FJS},
\[
 h_M^*(t)=\sum_{j\ge0}\binom{a_1-1}j\binom{a_2-1}j t^j.
\]
Set $d=\min(a_1-1,a_2-1)$ and $u=\max(a_1-1,a_2-1)$. For $0\le j\le d$, the normalized coefficient is $h_j^*/\binom dj=\binom uj>0$. The ratios $\binom u{j+1}/\binom uj=(u-j)/(j+1)$ decrease with $j$ for $0\le j<d$, so the normalized sequence is log-concave in the disconnected case as well.

Since $(j+1)(d-j+1)>j(d-j)$ for $1\le j<d$, positivity and \eqref{eq:ulc} imply strict log-concavity at every internal index. Therefore, the coefficient sequence is unimodal.
\end{proof}
\section*{Acknowledgements}
This work is supported by the Guangdong Basic and Applied Basic Research Foundation (Grant No. 2026A1515012237).

\section*{Declaration of AI use}
The author developed the main results, proofs and methods. OpenAI's Codex assisted with finding and testing candidate counterexamples, proofreading, and language editing. The author reviewed and verified the manuscript and takes full responsibility for its content.

\end{document}